\documentclass[11pt,english]{article}
\usepackage{mathtools}
\usepackage{mathrsfs}
\usepackage{graphicx}
\usepackage{graphics}
\usepackage{setspace}
\usepackage{color}
\usepackage{amsmath,amssymb,amsthm}
\usepackage{epsfig}
\usepackage{babel}
\usepackage{tikz}
\usepackage{tikz-cd}
\usetikzlibrary{arrows,shapes,automata,backgrounds,petri,positioning,scopes,matrix, calc, plothandlers,plotmarks,patterns}
\usepackage{enumerate}
\usepackage{float}
\usepackage[utf8]{inputenc}
\usepackage[T1]{fontenc}
\usepackage[document]{ragged2e}
 \usepackage{verbatim}
\usepackage{tikz-3dplot}
\usepackage{pgfplots}
\usepackage{tkz-graph}
\GraphInit[vstyle = Normal]
\tikzset{
  LabelStyle/.style = {minimum width = 2em, 
                        text = red, font = \bfseries },
  VertexStyle/.append style = { inner sep=2pt,
                                font = \Large\bfseries, fill},
  EdgeStyle/.append style = {->, bend left} }

\newtheorem{thm}{Theorem}[section]
\numberwithin{equation}{section} 
\numberwithin{figure}{thm} 
\theoremstyle{plain}
\newtheorem*{thm*}{Theorem}
\theoremstyle{definition}
\theoremstyle{plain}
\newtheorem{thm_A}{Theorem}
\newtheorem*{defn*}{Definition}

\theoremstyle{plain}
\newtheorem*{conj*}{Conjecture}
\newtheorem{conj_A}{Conjecture}
\theoremstyle{plain}

\theoremstyle{plain} 

\theoremstyle{plain}
\theoremstyle{definition}
\newtheorem{ex}[thm]{Example}
\theoremstyle{remark}

\theoremstyle{plain}

\theoremstyle{plain}

\theoremstyle{plain}
\newtheorem{lem}[thm]{Lemma}
\newtheorem*{lem*}{Lemma} 
\theoremstyle{definition}
\newtheorem{defn}[thm]{Definition}

\newtheorem*{acknowledgment*}{Addentum}

\theoremstyle{plain}
\newtheorem*{ex*}{Example}
\theoremstyle{plain}

\AtBeginDocument{
  
}
\begin{document}
\pgfdeclarelayer{background}
\pgfsetlayers{background,main}
\title{An approach to the  Herzog-Sch\"onheim conjecture using automata}
\author{Fabienne Chouraqui}
\date{}
\maketitle
\begin{abstract}
Let $G$ be a group and $H_1$,...,$H_s$ be subgroups of $G$ of  indices $d_1$,...,$d_s$ respectively. In 1974, M. Herzog and J. Sch\"onheim conjectured that if $\{H_i\alpha_i\}_{i=1}^{i=s}$,  $\alpha_i\in G$, is a coset partition of $G$, then $d_1$,..,$d_s$ cannot be distinct. In this paper, we present a   new approach to  the  Herzog-Sch\"onheim conjecture based on automata and present a translation of the conjecture as a problem on automata.
\end{abstract}
\maketitle
\section{Introduction}

Let $G$ be a group and $H_1$,...,$H_s$ be subgroups of $G$.  If there exist  $\alpha_i\in G$ such that $G= \bigcup\limits_{i=1}^{i=s}H_i\alpha_i$, and the sets  $H_i\alpha_i$, $1 \leq i \leq s$,  are pairwise disjoint, then  $\{H_i\alpha_i\}_{i=1}^{i=s}$ is \emph{a coset partition of $G$}  (or a \emph{disjoint cover of $G$}). In this case,    all the subgroups  $H_1$,...,$H_s$ can be assumed to be of  finite index in  $G$ \cite{newman,korec}. We denote by $d_1$,...,$d_s$ the indices of $H_1$,...,$H_s$ respectively. The coset partition $\{H_i\alpha_i\}_{i=1}^{i=s}$ has  \emph{multiplicity} if $d_i=d_j$ for some $i \neq j$. 

\setlength\parindent{10pt}	If $G$ is the infinite cyclic group $\mathbb{Z}$, a coset partition of $\mathbb{Z}$ is 
$\{d_i\mathbb{Z} +r_i\}_{i=1}^{i=s}$, $r_i \in \mathbb{Z}$,  with  
each $d_i\mathbb{Z} +r_i$ the residue class of $r_i$ modulo $d_i$.
These coset partitions of $\mathbb{Z}$ were first introduced by P. Erd$\ddot{o}$s \cite{erdos1} and he conjectured that if $\{d_i\mathbb{Z} +r_i\}_{i=1}^{i=s}$, $r_i \in \mathbb{Z}$, is a  coset partition of $\mathbb{Z}$, then  the largest index $d_s$ appears at least twice.  Erd$\ddot{o}$s' conjecture was 
proved independently by H. Davenport with R.Rado and L. Mirsky with  D. Newman using analysis of complex function \cite{erdos2,newman,znam}. Furthermore, it was proved that  the largest index $d_s$ appears at least $p$ times,  where $p$ is the smallest prime dividing $d_s$ \cite{newman,znam,sun2}, that each index $d_i$  divides another index $d_j$, $j\neq i$, and  that each index $d_k$ that does not properly divide any other index  appears at least twice \cite{znam}. We refer also to \cite{por1,por2,por3,por4,sun3} for more details on  coset partitions of $\mathbb{Z}$ (also called covers of $\mathbb{Z}$ by arithmetic progressions) and to \cite{ginosar} for a proof of the Erd$\ddot{o}$s' conjecture using group representations. 

In 1974, M. Herzog and J. Sch$\ddot{o}$nheim extended Erd$\ddot{o}$s' conjecture for arbitrary groups and conjectured that if $\{H_i\alpha_i\}_{i=1}^{i=s}$,  $\alpha_i\in G$, is a  coset partition of $G$, then $d_1$,..,$d_s$ cannot be distinct. In the 1980's, in a series of papers,  M.A. Berger, A. Felzenbaum and A.S. Fraenkel studied  the Herzog-Sch$\ddot{o}$nheim conjecture \cite{berger1, berger2,berger3} and in \cite{berger4} they proved the conjecture is true for the pyramidal groups, a subclass of the finite solvable groups. Coset partitions of finite groups with additional assumptions on the subgroups of the partition have been extensively studied. We refer to \cite{brodie,tomkinson1, tomkinson2,sun}. In \cite{schnabel}, the authors very recently proved that the conjecture is true for all groups of order less than $1440$. 

	The common approach to the Herzog-Sch$\ddot{o}$nheim (HS) conjecture is to study it in finite groups. Indeed, given any group $G$, every coset partition of $G$ induces a coset partition of a  finite  quotient group of $G$ with the same indices \cite{korec}. In this paper, we present a completely  different approach to the   HS conjecture. The idea is  to  study it in  free groups of finite rank and  from there to provide answers for every group. This is possible since  any finite or finitely generated  group  is a quotient group  of a free group of finite rank and   any coset partition of a quotient group  $F/N$ induces a  coset partition of $F$ with the same indices \cite{chou_hs}.\\ 

In order  to  study the   Herzog-Sch$\ddot{o}$nheim conjecture in  free groups of finite rank, we  use the machinery of covering spaces. A pair $(\tilde{X},p)$ 
is a  \emph{covering space} of  a topological space $X$ if 
$\tilde{X}$ is a path connected  space, 
 $p: \tilde{X}\rightarrow X$ is an open continuous surjection and 
 every $x \in X$ has an open neighborhood $U_x$ such that $p^{-1}(U_x)$ is a disjoint union of  open sets in $\tilde{X}$, each of which is mapped homeomorphically onto $U_x$ by $p$.
 For each $x \in X$, the non-empty set $Y_x=p^{-1}(x)$ is called \emph{the fiber over $x$} and for all $x,x' \in X$, $\mid Y_x \mid =\mid Y_{x'}\mid $. If the cardinal of a fiber is $m$,  one  says that  $(\tilde{X},p)$  is a \emph{$m$-sheeted covering} ($m$-fold cover) of $X$ \cite{hatcher,rotman}.\\
 
  The  fundamental group of the  bouquet with $n$ leaves (or the wedge sum of $n$ circles),  $X$,  is $F_n$, the  free group of finite rank $n$ and for any  subgroup $H$ of $F_n$ of finite index $d$, there exists  a $d$-sheeted covering space  $(\tilde{X}_H,p)$  with a fixed basepoint. The underlying graph 
of $\tilde{X}_H$  is a directed labelled graph, with $d$ vertices, called \emph{the Schreier graph} and it t can be seen as  a finite complete bi-deterministic automaton; fixing the start and the end state at the basepoint, it recognises the set of elements in $H$.  It is   called \emph{the Schreier coset diagram  for $F_n$ relative to the subgroup  $H$} \cite[p.107]{stilwell} or  \emph{the Schreier automaton for $F_n$ relative to the subgroup $H$} \cite[p.102]{sims}. The $d$ vertices (or states) correspond to the $d$ right cosets of $H$,  any edge (or transition) has the form $Hg \xrightarrow{a}Hga$, $g \in F_n$, $a$ a generator of $F_n$, and it describes the right action of $F_n$ on the right cosets of $H$.   If we fix the start state at   the basepoint ($H$),  and the end state at another vertex  $H\alpha$, where $\alpha$  denotes the label of some path from the start state to the end state, then this automaton recognises the set of elements in $H\alpha$ and we call it  \emph{the  Schreier automaton  of $H\alpha$} and denote it by $\tilde{X}_{H\alpha}$.\\

In general, for any  automaton $M$, with alphabet $\Sigma$, and $d$ states,  there exists a  square matrix $A$ of order $d\times d$, with $a_{ij}$ equal to the number of directed edges from  vertex $i$ to vertex $j$, $1\leq i,j\leq d$. This matrix is   non-negative and it is  called  \emph{the transition matrix} \cite{epstein-zwik}. If for every $1\leq i,j\leq d$, there exists $m \in \mathbb{Z}^+$ such that $(A^m)_{ij}>0$, the matrix is \emph{irreducible}. For $A$  an irreducible non-negative matrix, \emph{the period of $A$} is  the gcd of all $m \in \mathbb{Z}^+$ such that $(A^m)_{ii} >0$ (for any $i$).  If  $i$ and $j$ denote respectively the start and end  states of $M$, then the number of  words of length $k$ (in the alphabet $\Sigma$) accepted by $M$ is $a_k=(A^k)_{ij}$. \emph{The generating function of $M$} is defined by  $p(z)=\sum\limits_{k=0}^{k=\infty}a_k\,z^k$.  It is a rational function: the fraction of two polynomials in $z$ with integer coefficients \cite{epstein-zwik}, \cite[p.575]{stanley}.\\

 In \cite{chou-davenport}, we study the properties of the transition matrices and  generating functions of the Schreier automata in the context of coset partitions of the free group.  Let $F_n=\langle \Sigma\rangle$, and $\Sigma^*$ the free monoid  generated by $\Sigma$. Let $\{H_i\alpha_i\}_{i=1}^{i=s}$ be a coset  partition of $F_n$  with $H_i<F_n$ of index $d_i>1$, $\alpha_i \in F_n$, $1 \leq i \leq s$. Let $\tilde{X}_{i}$ denote the  Schreier  graph  of $H_i$, with transition matrix $A_i$ of period $h_i\geq 1$ and $\tilde{X}_{H_i\alpha_i}$  the  Schreier  automaton of $H_i\alpha_i$, with generating function  $p_i(z)$, $1 \leq i\leq s$. For each $\tilde{X}_{i}$,  $A_i$ is a non-negative irreducible matrix and $(A_i^{k})_{bf}$, $k \geq 0$, counts the  number of  words of length $k$ that belong to $H_i\alpha_i\cap \Sigma^*$ (with $b$ and $f$  denoting the start and end state  of $H_i\alpha_i$ respectively).
 Since $F_n$ is the disjoint union of the sets  $\{H_i\alpha_i\}_{i=1}^{i=s}$, each element  in $\Sigma^*$ belongs to one and exactly one such set, so $n^k$, the number of  words of length $k$ in $\Sigma^*$, satisfies $n^k=\sum\limits_{i=1}^{i=s}a_{i,k}$, for every $k \geq 0$, and moreover $\sum\limits_{k=0}^{k=\infty}n^k\,z^k=\sum\limits_{i=1}^{i=s}p_i(z)$. By using this kind of counting argument and studying  the behaviour of the generating functions at their poles, we prove that if $h=max\{h_i \mid 1 \leq i \leq s\}$ is greater than $1$, then there is a repetition of the maximal  period $h>1$ and   that,  under certain conditions, the coset partition has multiplicity. Furthermore, we 
  recover the  Davenport-Rado result (or Mirsky-Newman result) for the Erd\H{o}s' conjecture and some of its consequences. \\

In this paper, we deepen further our study of  the transition matrices of the Schreier automata in the context of coset partitions of $F_n$ and  give some new conditions that ensure a coset partition of $F_n$ has multiplicity.
\begin{thm_A}\label{theo_periods}
	 Let $F_n$ be the free group on $n \geq 2$ generators. Let $\{H_i\alpha_i\}_{i=1}^{i=s}$ be a coset  partition of $F_n$ with $H_i<F_n$ of index $d_i$, $\alpha_i \in F_n$, $1 \leq i \leq s$, and $1<d_1 \leq ...\leq d_s$.  Let $\tilde{X}_{i}$ denote the  Schreier  graph of $H_i$, with  transition matrix   $A_i$, and period $h_i \geq 1$, $1 \leq i\leq s$.
	 	 Let $H=\{h>1 \,\mid\,\exists 1 \leq j\leq s,\,h_j=h\}$. Assume  $H\neq \emptyset$ and different elements in  $H$ are pairwise coprime. Let $r_h=\mid \{1 \leq j\leq s,\mid\,h_j=h\}\mid $, the number of repetitions of $h$.
	 	 	 If for some $h \in H$, $r_h=h$ or $h<r_h\leq 2(h-1)$,  then $\{H_i\alpha_i\}_{i=1}^{i=s}$ has multiplicity.
	 	 
	\end{thm_A}

  Furthermore, we show the Herzog-Sch\"onheim conjecture in free groups can be translated into a  conjecture on automata.\\
  \begin{thm_A}
  If the following conjecture on automata is true:
   \begin{conj_A}
  	Let $\Sigma$ be a finite  alphabet, and $\Sigma^*$ be the free monoid  generated by $\Sigma$. For every $1 \leq i \leq s$, let $M_i$ be a finite,  bi-deterministic and complete automaton with strongly-connected underlying graph. Let  $d_i$ be the number of  states of $M_i$ ($d_i>1$), and  $L_i \subsetneq \Sigma^*$ be the accepted language of $M_i$. If   $\Sigma^*$ is equal to the disjoint union of the $s$ languages  $L_1, L_2,...,L_s$, then there are $1 \leq j,k \leq s$, $j\neq k$, such that $d_j=d_k$.
  \end{conj_A}  
   Then the   Herzog-Sch\"onheim conjecture   is true.
    \end{thm_A}
    







    The paper is organized as follows.  In  Section $2$, we give some preliminaries on automata and on irreducible non-negative matrices. In Section $3$, we present a particular class of automata adapted to the study of the  Herzog-Sch\"onheim conjecture in free groups and describe some of their properties.   In Section $4$, we prove Theorem $1$ and  Theorem $2$.  The last section is an appendix with the proof of Lemma \ref{prop_several-h-prime}.  We  refer to \cite{chou_hs} for more preliminaries and examples: Section 2, for free groups and covering spaces and   Section 3.1, for graphs.    
      
 \section{Automata, Non-negative irreducible matrices}

 \subsection{Automata}\label{subsec_automat}
 We refer the reader to \cite[p.96]{sims}, \cite[p.7]{epstein}, \cite{pin,pin2}, \cite{epstein-zwik}.
 A \emph{finite state automaton} is a quintuple $(S,\Sigma,\mu,Y,s_0)$, where $S$ is a finite set, called the \emph{state set}, $\Sigma$ is a finite set, called the \emph{alphabet}, $\mu:S\times \Sigma \rightarrow S$ is a function, called the \emph{transition function}, $Y$ is a (possibly empty) subset of $S$ called the \emph{accept (or end) states}, and $s_0$ is called the \emph{start state}.  It is a directed  graph with vertices the states and each transition $s \xrightarrow{a} s'$ between states $s$ and $s'$ is an edge with label $a \in \Sigma$. The \emph{label of a path $p$} of length $n$  is the product $a_1a_2..a_n$ of the labels of the edges of $p$.
 The  finite state automaton $M=(S,\Sigma,\mu,Y,s_0)$ is \emph{deterministic} if there is only one initial state and each state is the source of exactly one arrow with any given label from  $\Sigma$. In a deterministic automaton, a path is determined by its starting point and its label \cite[p.105]{sims}. It is \emph{co-deterministic} if there is only one final state and each state is the target of exactly one arrow with any given label from  $\Sigma$. The  automaton $M=(S,\Sigma,\mu,Y,s_0)$ is \emph{bi-deterministic} if it is both deterministic and co-deterministic. An automaton $M$ is \emph{complete} if for each state $s\in S$ and for each $a \in \Sigma$, there is exactly one edge from $s$ labelled $a$. We say that an automaton $M$ is \emph{strongly-connected} if there is a directed path from any state to any other state.
 \begin{defn}
 Let $M=(S,\Sigma,\mu,Y,s_0)$ be  a finite state automaton. Let $\Sigma^*$ be the free monoid generated by $\Sigma$. Let  $\operatorname{Map}(S,S)$ be  the monoid consisting of all maps from $S$ to $S$. The map $\phi: \Sigma \rightarrow \operatorname{Map}(S,S) $ given by $\mu$ can be extended in a unique way to a monoid homomorphism $\phi: \Sigma^* \rightarrow \operatorname{Map}(S,S)$. The range of this map is a monoid called \emph{the transition monoid of $M$}, which is generated by $\{\phi(a)\mid a\in \Sigma\}$. An element $w \in \Sigma^*$ is \emph{accepted} by $M$ if the corresponding element of $\operatorname{Map}(S,S)$, $\phi(w)$,  takes $s_0$ to an element of the accept states set $Y$. The set $ L\subseteq \Sigma^*$  recognized by $M$ is called \emph{the language accepted by $M$}, denoted by $L(M)$.
 \end{defn}
 For any directed  graph with $d$ vertices or any finite state  automaton $M$, with alphabet $\Sigma$, and $d$ states,  there exists a  square matrix $A$ of order $d\times d$, with $a_{ij}$ equal to the number of directed edges from  vertex $i$ to vertex $j$, $1\leq i,j\leq d$. This matrix is   non-negative (i.e $a_{ij}\geq 0$) and it is  called  \emph{the transition matrix} (as in \cite{epstein-zwik}) or  \emph{the adjacency matrix} (as in \cite[p.575]{stanley}). For any $k \geq 1$, $(A^k)_{ij}$ is  equal to the number of directed paths of length $k$   from  vertex $i$ to vertex $j$. So, if   $M$ is a bi-deterministic automaton  with alphabet $\Sigma$, $d$ states, start state $i$, accept state  $j$  and  transition matrix $A$, then   $(A^k)_{ij}$ is  the number of  words of length $k$ in the free monoid $\Sigma^*$  accepted by $M$. 
 \subsection{Irreducible non-negative matrices}
 We refer to \cite[Ch.16]{bellman}, \cite[Ch.8]{meier}. There is a vast literature on the topic. Let $A$ be a transition matrix  of order $d\times d$ of a directed graph or an automaton with $d$ states, as defined in Section 2.2. If for every $1\leq i,j\leq d$, there exists $m_{ij} \in \mathbb{Z}^+$ such that $(A^{m_{ij}})_{ij}>0$, the matrix is \emph{irreducible} and this is equivalent to the graph being strongly-connected.  For $A$  an irreducible non-negative matrix, \emph{the period of $A$} is  the gcd of all $m \in \mathbb{Z}^+$ such that $(A^m)_{ii} >0$ (for any $i$). If the period is $1$, A is called \emph{aperiodic}. In \cite{meier}, an   irreducible and  aperiodic matrix $A$ is  called \emph{primitive} and the period $h$ is called the \emph{index of imprimitivity}.\\
 
 Let A be an irreducible non-negative  matrix of order $d\times d$ with period $h\geq 1$ and spectral radius $r$. Then the Perron-Frobenius theorem states that $r$ is a positive real number and it is a simple eigenvalue of $A$,  $\lambda_{PF}$,  called the \emph{Perron-Frobenius (PF) eigenvalue}. It satisfies $\sum\limits_{i}a_{ij}\leq \lambda_{PF} \leq \sum\limits_{j}a_{ij}$. The matrix $A$ has a right eigenvector $v_R$ with eigenvalue $\lambda_{PF}$ whose components are all positive and likewise, a 
 left eigenvector $v_L$ with eigenvalue $\lambda_{PF}$ whose components are all positive.  Both right and left eigenspaces associated with $\lambda_{PF}$ are one-dimensional. The behaviour of irreducible non-negative matrices depends strongly on whether the matrix is aperiodic or not. 
  \begin{thm}\label{theo_aperiodic}\cite[Ch.8]{meier}
  Let $A$ be  a $d \times d$ irreducible  non-negative matrix of period $h \geq 1$, with PF eigenvalue $\lambda_{PF}$. Let   $v_L$  and  $v_R$ be left and right eigenvectors of $\lambda_{PF}$ whose   components are all positive, with $v_L\,v_R=1$. \\
    If $h=1$,  $\lim\limits_{k \rightarrow\infty}\frac{A^k}{\lambda_{PF}^k}=P$, and if $h>1$, $\lim\limits_{k \rightarrow\infty}\frac{1}{k}\sum\limits_{m=0}^{m=k-1}\frac{A^m}{\lambda_{PF}^m}=P$;   $P=v_R\,v_L$. 
  \end{thm}    
 
 \section{A particular class of automaton adapted to the study of the HS conjecture}
 \subsection{The Schreier automaton of a coset of a subgroup}
   
  We now introduce the particular class of automata we are interested in, that is    \emph{the Schreier automaton for $F_n$ relative to the subgroup $H$} \cite[p.102]{sims}, \cite[p.107]{stilwell}. We refer to \cite{chou_hs} for concrete examples.
  \begin{defn}\label{def_Schreier-graph}
   Let $F_n=\langle\Sigma\rangle$ and  $\Sigma^*$ the free monoid generated by $\Sigma$.  Let $H<F_n$ of index $d$. Let $(\tilde{X}_H,p)$ be the covering  of the $n$-leaves bouquet with basepoint $\tilde{x}_1$ and  vertices  $\tilde{x}_1, \tilde{x}_2,...,\tilde{x}_{d}$. Let  $t_i \in \Sigma^*$ denote the label of a directed path of minimal length from $\tilde{x}_1$ to $\tilde{x}_i$. Let $\mathscr{T}=\{1, t_i\mid 1 \leq i\leq d\}$.  Let $\tilde{X}_H$ be the Schreier coset diagram  for $F_n$ relative to the subgroup  $H$,  with $\tilde{x}_1$ representing the subgroup $H$ and the other vertices $\tilde{x}_2,...,\tilde{x}_{d}$ representing the cosets  $Ht_i$ accordingly.  We call $\tilde{X}_{H}$  \emph{the  Schreier graph  of $H$},  with this correspondence between the vertices  $\tilde{x}_1, \tilde{x}_2,...,\tilde{x}_{d}$ and the cosets  $Ht_i$ accordingly.  
     \end{defn} 
     From its definition, $\tilde{X}_{H}$  is  a  strongly-connected  $n$-regular graph. So, its transition  matrix $A$ is non-negative and irreducible, with PF eigenvalue $n$ (the sum of the elements at each row and at each column  is equal to $n$). 
        
  \begin{defn}\label{def_Schreier-automaton}
   Let $F_n=\langle\Sigma\rangle$ and $\Sigma^*$ the free monoid generated by $\Sigma$.  Let $H<F_n$ of index $d$. Let  $\tilde{X}_{H}$  be the  Schreier graph  of $H$. Using the notation from Defn. \ref{def_Schreier-graph}, let $\tilde{x}_1$ be the start state and $\tilde{x}_f$ be the end state for some $1 \leq f \leq d$. We call  the automaton obtained  \emph{the  Schreier automaton   of $Ht_f$}   and denote it by  $\tilde{X}_{Ht_f}$. The language accepted by $\tilde{X}_{Ht_f}$ is  the set of elements in $\Sigma^*$  that belong to $Ht_f$. We call the elements in  $\Sigma^*\cap Ht_f$,  \emph{the positive words in $Ht_f$}. The identity may belong to  this set.
     \end{defn} 
  \begin{ex}\label{ex_automaton_index4}
   Let  $\Sigma=\{a,b\}$;  $F_2=\langle a, b \rangle$.  Let $K\leq F_2$, of index $4$. 
  \begin{figure}[H] 
     \centering \scalebox{0.7}[0.5]{\begin{tikzpicture}
     \SetGraphUnit{4}
      \tikzset{VertexStyle/.append  style={fill}}
       \Vertex[L=$K$, x=-3,y=0]{A}
       \Vertex[L=$Ka$, x=0, y=0]{B}
     
     \Vertex[L=$Ka^{2}$, x=3, y=0]{C}
      \Vertex[L=$Ka^{3}$, x=6, y=0]{D}
     \Edge[label = a, labelstyle = above](A)(B)
      \Edge[label = a, labelstyle = above](B)(C)
       \Edge[label = a, labelstyle = above](C)(D)
      \Edge[label = b, labelstyle = below](D)(A)
   \tikzset{EdgeStyle/.style = {->}}    
    \Edge[label =b, labelstyle = below](A)(B)
   \Edge[label = b, labelstyle = below](B)(C)
     \Edge[label = b, labelstyle = below](C)(D)
  
  \tikzset{EdgeStyle/.style = {->, bend right}} 
  \Edge[label = a, labelstyle = above](D)(A)
     \end{tikzpicture}}
     \caption{The Schreier graph $\tilde{X}_K$  of $K= \langle a^4,b^4,ab^{-1},a^2b^{-2},a^3b^{-3} \rangle$.}  \label{fig_aut2}
  \end{figure}
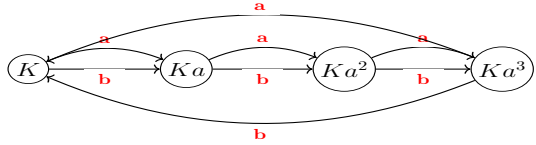The transition matrix of $\tilde{X}_K$ is  $ \left( \begin{array}{cccc}
                         0 & 2 & 0 & 0 \\
                          0 & 0 & 2 & 0\\
                          0 &  0 & 0 & 2\\
  						2 & 0  & 0 & 0
                           \end{array} \right)$ 
 with  period  $4$. If $K$ and $Ka$ are the start and end states, $L$ is  the set of positive words in $Ka$. 
  \end{ex}
  \subsection{Properties of the Schreier automata in coset partitions}
We recall here some results proved in \cite{chou-davenport}.  \begin{thm}\label{theo-repete-h}\cite{chou-davenport}
 Let $F_n$ be the free group on $n \geq 1$ generators. Let $\{H_i\alpha_i\}_{i=1}^{i=s}$ be a coset  partition of $F_n$ with $H_i<F_n$ of index $d_i$, $\alpha_i \in F_n$, $1 \leq i \leq s$, and $1<d_1 \leq ...\leq d_s$.  Let $\tilde{X}_{i}$ denote the  Schreier  graph of $H_i$, with  transition matrix   $A_i$, and period $h_i \geq 1$, $1 \leq i\leq s$. Let $1 \leq k,m\leq s$.
\begin{enumerate}[(i)]
\item Assume $h_k>1$, where  $h_k=max\{h_i \mid 1 \leq i \leq s\}$.   Then there exists $j\neq k$ such that  $h_j=h_k$. 
\item Let $h_{\ell}>1$, such that $h_{\ell}$ does not properly divide any other period $h_i$, $ 1 \leq i \leq s$.  Then there exists  $j\neq \ell$ such that  $h_j=h_{\ell}$.
\item  For every $h_i$, there exists  $j\neq i$ such that  either $h_i=h_j$ or $h_i\mid h_j$.
\end{enumerate}
  \end{thm} 

	If   $n=1$ in Theorem \ref{theo-repete-h}, $\{H_i\alpha_i\}_{i=1}^{i=s}$  is a coset  partition of $\mathbb{Z}$ and  we recover the  Davenport-Rado result (or Mirsky-Newman result) for the Erd\H{o}s' conjecture and some of its consequences. Indeed, for every index $d$, the Schreier graph of $d\mathbb{Z}$ has a transition matrix with period equal to $d$, so a repetition of the period is equivalent to a repetition of the index. For the  unique subgroup $H$ of $\mathbb{Z}$ of  index $d$,  its Schreier graph  $\tilde{X}_{H}$ is a closed directed path of length $d$ (with each edge labelled $1$). So, its  transition matrix   $A$ is the permutation matrix corresponding to the $d-$cycle $(1,2,...,d)$, and it has   period $d$.  In particular, the period of $A_s$ is $d_s$, and  there exists $j\neq s$ such that $d_j=d_s$.  Also, if the period (index) $d_k$ of $A_k$ does not properly divide any other period (index),  then  there exists  $j\neq k$ such that $d_j=d_k$. For the free groups in general, we prove that in some cases, the repetition of the period implies the repetition of the index (see \cite{chou-davenport}).

\section{Proof of the main results}

\subsection{Properties of the transition matrix of the Schreier graph}
We study the properties of  the transition matrix of a Schreier graph.
 \begin{lem}\label{lem_matrixP}
Let $H < F_n$ of index $d$, with Schreier graph $\tilde{X}_H$ and transition matrix $A$ with period $h\geq 1$.  Then the following   properties hold:
 \begin{enumerate}[(i)]
\item  The vectors  $v_L=\frac{1}{d}(1,1,...,1)$, $v_R=(1,1,...,1)^T$  are left and right eigenvectors of $n$ whose components are all positive, with $v_Lv_R=1$.
\item  The  matrix $P=v_R\,v_L$ is   of order $d\times d$ with all entries equal $\frac{1}{d}$.
 \item If $h=1$, then  $\lim\limits_{k \rightarrow\infty}\frac{A^k}{n^k}=P$ and if $h>1$, then  $\lim\limits_{k \rightarrow\infty}\frac{1}{k}\sum\limits_{j=1}^{j=k}\frac{A^j}{n^j}=P$.
 \end{enumerate}
  \end{lem}
 \begin{proof}
$(i)$, $(ii)$, $(iii)$  As the sum of every row and every column in $A$ is equal to $n$, $\lambda_{PF}=n$ with right eigenvalue  $v_R=(1,1,...,1)^T$ and left eigenvalue $(1,1,...,1)$. Since $(1,1,...,1)v_R=d$,  $v_L=\frac{1}{d}(1,1,...,1)$  is a left eigenvector that satisfies $v_Lv_R=1$. Computing  $v_R\,v_L$ gives    the matrix $P$  of order $d\times d$ with all entries equal $\frac{1}{d}$.  $(iii)$ results from Theorem \ref{theo_aperiodic}.
\end{proof}
The behaviour of exponents of an aperiodic $d \times d$ matrix of a Schreier graph $\tilde{X}_H$ is well known:  for every $1 \leq i,j \leq d$, $\lim\limits_{k \rightarrow\infty}\frac{(A^k)_{ij}}{n^k}=\frac{1}{d}$, from  Lemma \ref{lem_matrixP}. It means that the proportion of positive words of every length $k$ ($k$ large enough) that belong to any coset of $H$ tends to the fixed value $\frac{1}{d}$. We turn now to the study of  $\lim\limits_{k \rightarrow\infty}\frac{(A^k)_{ij}}{n^k}$, where $A$ is the transition matrix of a Schreier graph $\tilde{X}_H$ of period $h> 1$.
 \begin{defn}
 	For $1 \leq k,l \leq d$, we define  $m_{ij}$,  $0 \leq m_{ij} \leq h$, to  be the  minimal natural number  such that $(A^{ m_{ij}})_{ij} \neq 0$.
 \end{defn} 
 By definition,  if $i \neq j$, then $m_{ij}$ is  the  minimal length of a directed  path from   $i$ to $j$ in $\tilde{X}_H$ and if $i=j$, then $m_{ij}=0$. Whenever $h>1$,  only for the exponents $m_{ij}+kh$, $k\geq 0$,  $(A^{ m_{ij}+kh})_{ij} \neq 0$, that is  only positive words of length $m_{ij}+kh$ are accepted by the Schreier automaton, with  $i$ and $j$  the start and end states respectively.   Note that if $H$ is a subgroup of $\mathbb{Z}=\langle 1 \rangle$ of index $d$,  its transition matrix $A$ is a permutation matrix with  period $d$ and $m_{ij}=r$, where $d\mathbb{Z}+r$ is the coset with  $i$ and $j$  the start and end states respectively.  
  \begin{lem}\label{lem-period}
   Let $H < F_n$ of index $d$, with Schreier graph $\tilde{X}_H$ and transition matrix $A$ with period $h> 1$. Then,  the following   properties hold:
   \begin{enumerate}[(i)]
   \item $\lim\limits_{k\rightarrow\infty}\frac{(A^{k})_{ij}}{n^{k}}=0$,  whenever 
  $k \not\equiv  m_{ij}(mod\,h)$,  $1 \leq i,j \leq d$.
       \item  $\lim\limits_{k\rightarrow\infty}\frac{(A^{k})_{ij}}{n^{k}}=\frac{h}{d}$, whenever   $k \equiv  m_{ij}(mod\,h)$,  $1 \leq i,j \leq d$.
       	\item for every $0 \leq m \leq h-1$, there is $i$ such that $m_{1i}\equiv m(mod\, h)$.
         \item  $h$ divides $d$. 
         
    \end{enumerate}
   \end{lem}
 \begin{proof}
  $(i)$ By definition, whenever $k \not\equiv m_{ij}(mod\,h)$, $(A^{k})_{ij}=0$.   \\
  $(ii)$,  $(iii)$, $(iv)$   We define a $d \times h$ matrix $B$ in the following way. Each row $i$ is labelled by  a right coset of $H$  in the same order as they appear in the rows and columns of $A$ and each column by  $m=0,1,2,...,h-1$, and $(B)_{ij} =  \lim\limits_{k\rightarrow\infty}\frac{(A^{j+kh})_{1i}}{n^{j+kh}}$. Roughly, $(B)_{ij}$ is the proportion of positive words of very large length that belong to the corresponding coset of $H$. From $(i)$: 
    \begin{equation*}
       (B)_{ij} = \begin{cases}
                 0 & \text{if  } m_{1i}\not\equiv j (mod \,h),\\
                 \lim\limits_{k\rightarrow\infty}\frac{(A^{m_{1i}+kh})_{1i}}{n^{m_{1i}+kh}} & \text{if } m_{1i}\equiv j (mod \,h).
             \end{cases}
   \end{equation*}
    So, at each row of $B$, there is a single non-zero entry. As $F_n$ is partitioned by the $d$ cosets of $H$, all the non-zero elements in $B$ are equal and for every $k\geq 0$ and every $1 \leq i \leq d$, $\sum\limits_{f=1}^{f=d}(A^k)_{if}=n^k$, in particular   $\sum\limits_{f=1}^{f=d}\frac{(A^k)_{1f}}{n^k}=1$. So,
    $\sum\limits_{i=1}^{i=d}(B)_{ij}=\sum\limits_{i=1}^{i=d} \lim\limits_{k\rightarrow\infty}\frac{(A^{j+kh})_{1i}}{n^{j+kh}}=$
    $\lim\limits_{k\rightarrow\infty}\sum\limits_{i=1}^{i=d} \frac{(A^{j+kh})_{1i}}{n^{j+kh}}=1$, that is the sum of elements in each column  of $B$ is equal to $1$.
     If $h=d$, $B$ is a square matrix and the right cosets can be arranged such that their $m$ is in growing order and we have necessarily a diagonal matrix (otherwise there would be a column of zeroes).  So, $ \lim\limits_{k\rightarrow\infty}\frac{(A^{m_{1i}+kh})_{1i}}{n^{m_{1i}+kh}} =1$ and $(ii),(iii),(iv)$ hold. Now, assume $d>h$. At each column, there is at least one non-zero entry, so $(iv)$ holds.  Furthermore, the number of non-zero entries in each column needs to be the same, so  $h$ divides $d$ and for any $i$, 
    $ \frac{d}{h}*(\lim\limits_{k\rightarrow\infty}\frac{(A^{m_{1i}+kh})_{1i}}{n^{m_{1i}+kh}})=1$. That is, $\lim\limits_{k\rightarrow\infty}\frac{(A^{m_{1i}+kh})_{1i}}{n^{m_{1i}+kh}}=\frac{h}{d}$. Furthermore,  $\lim\limits_{k\rightarrow\infty}\frac{1}{h}\sum\limits_{j=0}^{j=h-1}\frac{(A^{j+kh})_{1i}}{n^{j+kh}}=
    \frac{1}{d}$.
   \end{proof}

  \subsection{Conditions that ensure  multiplicity in a coset partition}
   Let $F_n$ be the free group on $n \geq 1$ generators. Let $\{H_i\alpha_i\}_{i=1}^{i=s}$ be a coset  partition of $F_n$ with $H_i<F_n$ of index $d_i$, $\alpha_i \in F_n$, $1 \leq i \leq s$, and $1<d_1 \leq ...\leq d_s$.  Let $\tilde{X}_{i}$ denote the  Schreier  graph of $H_i$, with  transition matrix   $A_i$, and period $h_i \geq 1$, $1 \leq i\leq s$.  In the following lemmas, we prove,  under these assumptions, there exist conditions that ensure multiplicity.
\begin{lem}\label{prop_one-h}
	   Assume there exists a unique  $h>1$. Let $r$ denote the number of repetitions of $h$.
	 Then, $r\geq h$. Furthermore, if $r=h$ of if $h<r\leq 2(h-1)$,  then $\{H_i\alpha_i\}_{i=1}^{i=s}$ has multiplicity.
	\end{lem}
	\begin{proof}
For every $1\leq i \leq s$,  $(A_i^k)_{1f_i}$ denotes the number of positive words of length $k$ that belong to the coset $H_i\alpha_i$. 	Let $I=\{1 \leq i\leq s \mid h_i=h\}$ and for every $i \in I$,  $m_i$ the minimal natural number $(mod \,h)$ such that $(A_i^{m_i})_{1f_i}\neq 0$.  We define a $r \times h$ matrix $C$ in the following way. Each row $i$ is labelled by  a right coset $H_i\alpha_i$, where $i \in I$  and each column by  $m=0,1,2,...,h-1$, and: 
	   \begin{equation*}
	      (C)_{ij} = \begin{cases}
	                0 & \text{if  } m_{i}\not\equiv j (mod \,h),\\
	                \lim\limits_{k\rightarrow\infty}\frac{(A_i^{m_i+hk})_{1f_i}}{n^{m_i+hk}} & \text{if } m_{i}\equiv j (mod \,h).
	            \end{cases}
	  \end{equation*}
	  Roughly, $(C)_{ij}$ is the proportion of positive words of very large length that belong to  $H_i\alpha_i$, where $i \in I$. At each row of $C$ there is a unique non-zero entry. 
	  Since $\{H_i\alpha_i\}_{i=1}^{i=s}$ is  a coset  partition of $F_n$, for very $k$,  $\sum\limits_{i=1}^{i=s}(A_i^k)_{1f_i}=n^k$, that is $\sum\limits_{i=1}^{i=s}\frac{(A_i^k)_{1f_i}}{n^k}=1$. If $A_i$ is  aperiodic, then $\lim\limits_{k\rightarrow\infty}\frac{(A_i^k)_{1f_i}}{n^k}=\frac{1}{d_i}$ from Lemma \ref{lem_matrixP}. So, 
	  $1= \lim\limits_{k\rightarrow\infty}\sum\limits_{i=1}^{i=s}\frac{(A_i^k)_{1f_i}}{n^k}=
	  \sum\limits_{i=1}^{i=s}\lim\limits_{k\rightarrow\infty}\frac{(A_i^k)_{1f_i}}{n^k}=
	  \sum\limits_{i \notin I}\frac{1}{d_i}+\sum\limits_{i \in I}\lim\limits_{k\rightarrow\infty}\frac{(A_i^k)_{1f_i}}{n^k}$.
	  That is, $\sum\limits_{i \in I}\lim\limits_{k\rightarrow\infty}\frac{(A_i^k)_{1f_i}}{n^k}=1-\sum\limits_{i \notin I}\frac{1}{d_i}=\sum\limits_{i \in I}\frac{1}{d_i}$, since  $\sum\limits_{i=1}^{i=s}\frac{1}{d_i}=1$. So,  the sum of elements in each column of $C$ is equal to $\sum\limits_{i\in I}\frac{1}{d_i}$ and  from Lemma \ref{lem-period},  the non-zero entries  in $C$  have the form $\frac{h}{d_i}$. If $r<h$, then there is necessarily a column of zeroes, so $r \geq h$. If $r=h$, then $C$ is a square matrix and the right cosets can be arranged such that their $m$ is in growing order and we have necessarily a diagonal matrix (otherwise there would be a column of zeroes).  So, for every $i \in I$, $ \frac{h}{d_i}= \sum\limits_{i\in I}\frac{1}{d_i}$. That is,   the coset partition $\{H_i\alpha_i\}_{i=1}^{i=s}$ has multiplicity with all the $d_i$ equal for $i \in I$.   
	    Now, assume $r>h$. At each column, there is at least one non-zero entry and there are necessarily columns with several non-zero entries. By some simple combinatorics, $n_0$, the number of columns with a single non-zero entry  satisfies  $h-(r-h)\leq n_0\leq h-1$,  that is $2h-r \leq n_0 \leq h-1$. If  we assume $r\leq 2(h-1)$, then $n_0\geq 2h-r-2(h-1) \geq 2$, that is  the number of columns with a single non-zero entry is at least  $2$, so there are at least two $i \in I$, such that  $ \frac{h}{d_i}= \sum\limits_{i\in I}\frac{1}{d_i}$, that is the coset partition $\{H_i\alpha_i\}_{i=1}^{i=s}$ has multiplicity.   Note that for every $0 \leq m\leq h-1$, there is $i$ such that $m_i\equiv m (mod \,h)$.	
	  \end{proof}
	  
	 	  \begin{lem}\label{prop_several-h-prime}
	 	     Assume there exists two   $h,h'>1$ and $h$ and $h'$ are coprime. 	Let $I=\{1 \leq i\leq s \mid h_i=h\}$,  $r=\mid I \mid$; $I'=\{1 \leq i\leq s \mid h_i=h'\}$, $r'=\mid I'\mid$. 
	 	  If $r=h$ or $r\leq 2(h-1)$ or $r'=h'$ or  $r'\leq 2(h'-1)$,  then $\{H_i\alpha_i\}_{i=1}^{i=s}$ has multiplicity.
	 	  \end{lem}
The proof of Lemma \ref{prop_several-h-prime} appears in the appendix. We prove there that  coprime periods  can be considered independently, that is each period can be assigned  its own matrix $C$ as defined in the proof of Lemma \ref{prop_one-h}.  The situation is different if $h$ and  $h'$ are not coprime. Indeed, consider the following  coset partition of $F_2$: $F_2=H\cup Ka\cup Ka^3$, where $K$ is the subgroup described in Example \ref{ex_automaton_index4} and $H=\langle a^2,b^2,ab\rangle\ < F_2$ of index $2$. The period of  $\tilde{X}_{H}$ is $h'=2$ and     the period of  $\tilde{X}_{K}$ is $h=4$  and the corresponding  matrix $D$ as defined in the proof of Lemma \ref{prop_several-h-prime} is $D=$  $ \left( \begin{array}{cccc}
	 	  	   	 	  	                            1 & 0 & 1 & 0 \\
	 	  	   	 	  	                             0 & 1 & 0 & 0\\
	 	  	   	 	  	                             0 &  0 & 0 & 1
	 	  	   	 	  	           \end{array} \right)$, with  the first row 
	 	  	   	 	  	           labelled $H$, the second row $Ka$,  the third row $Ka^3$ and at each column $0\leq m\leq 3$. So, if $h'$ divides $h$, each period cannot have  its own matrix $C$. Yet, using the same kind of arguments as before, it is not difficult to prove that $r \geq h-\frac{h}{h'}r'$ and that if $r \leq 2(h-\frac{h}{h'}r'-1)$ then the coset partition has multiplicity. We now turn to the proof of Theorem $1$. 	
	 	  	   \begin{proof}[Proof of Theorem \ref{theo_periods}]
	 	  		 We assume  $H=\{h>1 \,\mid\,\exists 1 \leq j\leq s,\,h_j=h\}\neq \emptyset$ and different elements in  $H$ are pairwise coprime. Let $r_h=\mid \{1 \leq j\leq s,\,h_j=h\}\mid $, the number of repetitions of $h$. From the proof of Lemma \ref{prop_several-h-prime},  each period has its own matrix $C$ and we apply the results of Lemma \ref{lem-period}. That is, if for some $h \in H$, $r_h=h$ or $h<r_h\leq 2(h-1)$,  then $\{H_i\alpha_i\}_{i=1}^{i=s}$ has multiplicity.
	 	  	   \end{proof}

\subsection{Translation of the HS conjecture in terms of automata}

 Let $F_n=\langle \Sigma\rangle$, and $\Sigma^*$ the free monoid  generated by $\Sigma$. Let $\{H_i\alpha_i\}_{i=1}^{i=s}$ be a coset  partition of $F_n$  with $H_i<F_n$ of index $d_i>1$, $\alpha_i \in F_n$, $1 \leq i \leq s$. Let $\tilde{X}_{i}$ be the  Schreier  automaton of $H_i\alpha_i$, with  language $L_i=\Sigma^*\cap H_i\alpha_i$.  
  \begin{proof}[Proof of Theorem $2$]
 Assume  Conjecture $1$ is true. For every $1 \leq i \leq s$, the  Schreier  automaton $\tilde{X}_{i}$ is a finite,  bi-deterministic and complete automaton with strongly-connected underlying graph and alphabet $\Sigma$.  Since $F_n$ is the disjoint union of the sets  $\{H_i\alpha_i\}_{i=1}^{i=s}$, each word  in $\Sigma^*$ belongs to one and exactly one such language, so  $\Sigma^*$ is the disjoint union of the $s$ languages  $L_1, L_2,...,L_s$. Since Conjecture $1$ is true, there is a repetition of the number of states and this implies the coset  partition $\{H_i\alpha_i\}_{i=1}^{i=s}$ has multiplicity, that is the  HS conjecture in free groups of finite rank is  true. From \cite[Thm.6]{chou_hs}, this implies the HS conjecture is true for all the  finitely generated groups, in particular for all the finite groups. So,  the  HS conjecture is true for all the groups.
  \end{proof}
Note that these two  conjectures would have been  equivalent if the conditions of Conjecture 1 implied  the existence of a coset partition of the free group of rank $\mid\Sigma\mid$, which does not seem to be true. Nevertheless, any finite,  bi-deterministic,  complete and  strongly-connected automaton $M$ with $d$ states  can be considered as  the  Schreier  automaton of a subgroup $H$ of index $d$ in $F_{\mid\Sigma\mid}$. 

 

  \newpage
\section*{Appendix: Proof of Lemma \ref{prop_several-h-prime}}
   \begin{lem*}
 	 	     Assume there exists two   $h,h'>1$ and $h$ and $h'$ are coprime. 	Let $I=\{1 \leq i\leq s \mid h_i=h\}$,  $r=\mid I \mid$; $I'=\{1 \leq i\leq s \mid h_i=h'\}$, $r'=\mid I'\mid$. 
 	 	  If $r=h$ or $r\leq 2(h-1)$ or $r'=h'$ or  $r'\leq 2(h'-1)$,  then $\{H_i\alpha_i\}_{i=1}^{i=s}$ has multiplicity.
 	 	  \end{lem*}
 
 \begin{proof}[Proof of Lemma \ref{prop_several-h-prime}]
 	 	  	Assume with no loss of generality that $h' <h$. From the same argument as in the proof of Lemma \ref{prop_one-h}, $\sum\limits_{i \in I\cup I'}\lim\limits_{k\rightarrow\infty}\frac{(A_i^k)_{1f_i}}{n^k}=\sum\limits_{i \in I \cup I'}\frac{1}{d_i}$. We show that each period can be considered independently, that is each period has its own matrix $C$ as defined in the proof of Lemma \ref{prop_one-h}. 	We define a $(r'+r) \times L$ matrix $D$, 	where  $L=2hh'$,   in the following way. The first $r'$  rows  are  labelled by   right cosets $H_i\alpha_i$, where $i \in I'$, the last  $r$  rows  are  labelled by   right cosets $H_i\alpha_i$, where $i \in I$   and each column by  $m=0,1,2,...,h'-1,..,h-1,h,...,L-1$, and: 
 	 	  	\begin{equation*}
 	 	  	(D)_{ij} = \begin{cases}
 	 	  	0 & \text{if  } i\in I',\,  m_i\not\equiv j (mod\,h') ,\\
 	 	  		\lim\limits_{k\rightarrow\infty}\frac{(A_i^{m_i+h'k})_{1f_i}}{n^{m_i+h'k}} & \text{if } i\in I',\,  m_i\equiv j (mod\,h')\\
 	 	  			0 & \text{if  } i\in I,\, m_i\not\equiv j (mod\,h) ,\\
 	 	  	\lim\limits_{k\rightarrow\infty}\frac{(A_i^{m_i+hk})_{1f_i}}{n^{m_i+hk}} & \text{if } i\in I,\, m_i\equiv j (mod\,h).
 	 	  	 	  	\end{cases}
 	 	  	\end{equation*}
 	 	  	So,  the sum of elements in each column of $D$ is equal to $\sum\limits_{i \in I\cup I'}\frac{1}{d_i}$ and  from Lemma \ref{lem-period},  the non-zero entries  in $D$  have the form $\frac{h}{d_i}$ for $i \in I$ and  $\frac{h'}{d_i}$, for $i \in I'$. Let  $0\leq m \leq h'-1$, the minimal number such that   the sum of the entries of the $m$-th column is    $\sum\limits_{i\in J_0}\frac{h}{d_i}+\sum\limits_{i\in J_0'}\frac{h'}{d_i}$, where $\emptyset\neq J_0\subset I$ and $\emptyset\neq J_0' \subset I'$.  So, for every $ 0 \leq k \leq h'-1$,  the sum of the entries of the $(m+kh)$-th column is    $\sum\limits_{i\in J_0}\frac{h}{d_i}+\sum\limits_{i\in J_k'}\frac{h'}{d_i}$, and this implies  necessarily   $\sum\limits_{i\in J_0'}\frac{h'}{d_i}=\sum\limits_{i\in J_1'}\frac{h'}{d_i}=...=\sum\limits_{i\in J_{h'-1}'}\frac{h'}{d_i}$. We show that 
 	 	  	  $\{\frac{h'}{d_i}\mid i\in J_0'\}$,  $\{\frac{h'}{d_i}\mid i\in J_1'\}$,..., $\{\frac{h'}{d_i}\mid i\in J_{h'-1}'\}$ appear in the first $h'$ columns of $D$ (not necessarily in this order).	 Let  $ 0 \leq k,l \leq h'-1$, $k\neq l$. Assume by contradiction that $m+kh\equiv m+lh\, (mod\,h')$. So, $h'$ divides $h(k-l)$. As  	$h$ and $h'$ are coprime, $h'$ divides $k-l$, a contradiction. So, for every   $ 0 \leq k,l \leq h'-1$, $k\neq l$,  $m+kh\not\equiv m+lh\, (mod\,h')$. As there are exactly $h'$ values, these correspond to $0,1,...,h'-1\,(mod\,h')$, and  $\{\frac{h'}{d_i}\mid i\in J_0'\}$,  $\{\frac{h'}{d_i}\mid i\in J_1'\}$,..., $\{\frac{h'}{d_i}\mid i\in J_{h'-1}'\}$ appear in the first $h'$ columns of $D$ with $\sum\limits_{i\in J_0'}\frac{h'}{d_i}=...=\sum\limits_{i\in J_{h'-1}'}\frac{h'}{d_i}$. Furthermore, $\sum\limits_{i\in J_0'}\frac{h'}{d_i}=...=\sum\limits_{i\in J_{h'-1}'}\frac{h'}{d_i}=\sum\limits_{i\in I'}\frac{1}{d_i}$. Indeed, on one hand, the sum of elements in the first $r'$ rows and $h'$ columns  is equal to  $h'\,\sum\limits_{i\in J_0'}\frac{h'}{d_i}$ and on the second hand, it is equal $\sum\limits_{i\in I'}\frac{h'}{d_i}$.
 	 	  	 	  Using the same argument,  for every $ 0 \leq k \leq h-1$,  the sum of the entries of the $(m+kh')$-th column is    $\sum\limits_{i\in J_k}\frac{h}{d_i}+\sum\limits_{i\in J_0'}\frac{h'}{d_i}$, and this implies  necessarily    	  
 	 	  	  	 	  	  $\sum\limits_{i\in J_0}\frac{h}{d_i}=\sum\limits_{i\in J_1}\frac{h}{d_i}=...=\sum\limits_{i\in J_{h-1}}\frac{h}{d_i}$. We show that 
 	 	  	  	 	  $\{\frac{h}{d_i}\mid i\in J_0\}$,  $\{\frac{h}{d_i}\mid i\in J_1\}$,..., $\{\frac{h}{d_i}\mid i\in J_{h-1}\}$ 
 	 	  	  	 	   appear in the first $h$ columns of $D$ (not necessarily in this order). Let  $ 0 \leq k,l \leq h-1$, $k\neq l$. Assume by contradiction that $m+kh'\equiv m+lh'\, (mod\,h)$. So, $h$ divides $h'(k-l)$. As  	$h$ and $h'$ are coprime, $h$ divides $k-l$, a contradiction. So, for every   $ 0 \leq k,l \leq h-1$, $k\neq l$,  $m+kh'\not\equiv m+lh'\, (mod\,h)$. As there are exactly $h$ values, these correspond to $0,1,...,h-1(mod\,h)$, and  $\{\frac{h}{d_i}\mid i\in J_0\}$,  $\{\frac{h}{d_i}\mid i\in J_1\}$,..., $\{\frac{h}{d_i}\mid i\in J_{h-1}\}$  appear in the first $h$ columns of $D$, with $\sum\limits_{i\in J_0}\frac{h}{d_i}=\sum\limits_{i\in J_1}\frac{h}{d_i}=...=\sum\limits_{i\in J_{h-1}}\frac{h}{d_i}$. Furthermore, $\sum\limits_{i\in J_0}\frac{h}{d_i}=...=\sum\limits_{i\in J_{h-1}}\frac{h}{d_i}=\sum\limits_{i\in I}\frac{1}{d_i}$.
 	 	  	  	 	  	 So,  each period has its own matrix $C$ and we apply the results of Lemma \ref{lem-period}.
 	 	  	   \end{proof}

\end{document}